\documentclass[12pt]{article}
\usepackage{amsfonts}
\usepackage{amsmath}
\usepackage{amsxtra}
\usepackage{amssymb,latexsym}
\usepackage[mathcal]{eucal}
\usepackage{amscd}

\newtheorem{theo}{{\bfseries Theorem}}[section]
\newtheorem{prop}[theo]{{\bfseries Proposition}}

\newtheorem{lem}[theo]{{\bfseries Lemma}}
\newtheorem{cor}[theo]{{\bfseries Corollary}}
\newtheorem{df}[theo]{{\bfseries Definition}}

\newtheorem{ex}[theo]{{\bfseries Example}}

\def \ol {\overline}
\def \N {\mathbb N}

\def \Z {\mathbb Z}
\def \R {\mathbb R}
\def \A {\mathcal A}

\def \M {\mathcal M}

\def \O {\mathcal O}

\def \S {\mathcal S}

\def \V {\mathcal V}

\def \ep {\epsilon}

\def \d {\delta}

\def \s {\sigma}
\def \o {\omega}

\usepackage{amssymb,latexsym}
\usepackage[mathcal]{eucal}
\usepackage{amscd}
\usepackage{amsmath}
\usepackage{graphicx}
\usepackage{amsfonts}
\usepackage{amscd}

\numberwithin{equation}{section}

\begin{document}

\title{\bfseries  Variations on the Concept of Topological Transitivity}
\vspace{1cm}
\author{Ethan Akin, Joseph Auslander, Anima Nagar\\ \\
Mathematics Department,\\
 The City College, 137 Street and Convent Avenue,\\
 New York City, NY 10031, USA\\
ethanakin@earthlink.net\\ \\
Mathematics Department,\\
University of Maryland,\\
College Park, MD 20742, USA\\
jna@math.umd.edu \\ \\
Department of Mathematics,\\
Indian Institute of Technology Delhi,\\
Hauz Khas, New Delhi 110016, INDIA\\
anima@maths.iitd.ac.in}

    \vspace{.2cm}
\date{January, 2016}

\vspace{.2cm} \maketitle

\begin{abstract}
We describe  various strengthenings of the concept of topological transitivity. Especially when one
departs from the family of invertible systems, a number of interesting properties arise. We present the
architecture of implications among ten reasonable notions of transitivity.

\end{abstract}

\thanks{\emph{keywords:} topological transitivity, strong transitivity, exact  systems,locally eventually onto, noninvertible
dynamical systems}


\section{Introduction}

For us a dynamical system $(X,f)$ is a pair with $X$  a compact metric space,   $f: X \to X$
a
continuous map and the dynamics given  by iteration.
While we don't explicitly assume it, our main interest is in systems $(X,f)$ with $f$ not invertible.

For subsets $A, B \subset X$ we define the \emph{ hitting time set}
$$N(A,B) \quad = \quad \{ n \in \N : f^n(A) \cap B \not= \emptyset \} \ = \ \{ n \in \N : A \cap f^{-n}(B) \not= \emptyset \},$$
where $\N$ is the set of positive integers.  We identify a singleton with its unique element, writing $N(x,B)$, for example,
for $N(\{ x \},B)$.  We call the system \emph{topologically transitive} when for every opene (= open and nonempty) set $U$
$\bigcup_{n=1}^{\infty} f^n(U)$ is dense in $X$, or, equivalently if for every opene pair
$U, V \subset X$, the set of hitting times $N(U,V)$ is nonempty. A point $x \in X$  is called a \emph{transitive point} when
for every opene $V \subset X$, the hitting time set $N(x,V)$ is non-empty. This is equivalent to saying that the
orbit $\O(x) = \{f^n (x) : n \in \mathbb{N} \}$ is dense in $X$. We denote by $Trans(f)$ the set of transitive points.
Transitivity is equivalent to a
kind of indecomposability of the phase space
with respect to the acting group.

There are a number of slightly different versions of topological
transitivity, which are surveyed in \cite{ac}. The one we have chosen implies, by compactness,
that $f$ is surjective.  Furthermore, either $X$ is finite, and so consists of a single periodic
orbit, or it is perfect, i.e. it has no isolated points.
Our definition serves to exclude compactifications of the translation map on $\N$ or on $\Z$ the set of integers.
On a perfect space all of the definitions agree.

Here we  study some stronger forms of transitivity
\begin{itemize}
\item[1.]  Topological Transitivity  (TT).
\item[2.]  Strong Transitivity (ST).
\item[3.]  Very Strong Transitivity (VST).
\item[4.]  Minimality (M).
\item[5.]  Weak Mixing (WM).
\item[6.]  Exact Transitivity (ET).
\item[7.]  Strong Exact Transitivity (SET).
\item[8.]  Strong Product Transitivity (SPT).
\item[9.]  Mixing, or Topological Mixing (TM).
\item[10.] Locally Eventually Onto (LEO).
\end{itemize}

These concepts are not new to literature. Parry \cite{parry} had defined ``Strongly transitive systems"
which we here call ``very strongly transitive systems". The concept of ``strongly transitive" that we
define here has been studied as property ``(B)" by Nagar and Kannan \cite{avksm} and
by Akin, Auslander and Nagar \cite{aan}. The concept
of ``Locally eventually onto" was introduced by Denker and Urbanskii  in \cite{du} where they called  such systems
``exact".  We follow Kaminski et al. \cite{kss} who use the label ``exactness" for a somewhat different property, and then
use ``locally eventually onto" for the earlier notion.

It is known that the properties of strongly transitive, very strongly transitive and locally eventually onto are observed in
the dynamics of piece-wise monotonic maps, subshifts of finite type, $\beta$-shifts, and Julia sets.

If we consider the induced dynamics on $2^X$, the space of non-empty closed subsets of
$X$ given the Hausdorff topology, then it has been shown in Akin, Auslander and
Nagar \cite{aan} that for the induced system the concepts ``strongly transitive",
 and  ``local eventually onto" are  equivalent (and so these are equivalent to very strongly transitive, strongly product
 transitive and strong exact transitive as well).  Furthermore, these occur exactly when the original system
$(X,f)$ is locally eventually onto.

 Our purpose here is to define these concepts and list various equivalent conditions and properties.
 In the process we will show that the following
 implications hold
\begin{equation}\label{01}
\begin{split}
LEO \Longrightarrow ET, \ TM \ \Longrightarrow \  WM \ \Longrightarrow \ TT; \hspace{2.5cm}\\
LEO \  \Longrightarrow \ SPT \ \Longrightarrow \ SET \Longrightarrow \ ET, ST  \ \Longrightarrow \ TT; \hspace{1cm} \\
LEO, M \  \Longrightarrow \ VST \Longrightarrow \ ST \Longrightarrow \ TT. \hspace{3cm}
\end{split}
\end{equation}
We will show also that the properties are distinct by showing that various reverse implications fail.

The reader may note the absence of any discussion of sensitive dependence on initial conditions, i.e. \emph{sensitivity}.
Any topologically transitive system $(X,f)$ which admits an equicontinuity point is  almost equicontinuous and uniformly rigid
and so $f$ is a homeomorphism, see \cite{aab} and \cite{gw}. Hence, the topologically
transitive, noninvertible systems which are our primary concern are
all sensitive.

\vskip .5cm

\textbf{Acknowledgement:} A part of this work was done when A. Nagar visited
University of Maryland and she gratefully acknowledges the hospitality of the Mathematics Department at College Park.

\vspace{1cm}

\section{Transitivity Properties}

For a dynamical system $(X,f)$ and a point $x \in X$ the \emph{orbit} of $x$ is $\O(x) \ = \ \{f^n (x) : n \in \N \}$, regarded
either as a sequence in, or as a subset of, the state space $X$. Notice that $x = f^0(x)$ need not be an element of
$\O(x)$. We let $\o(x)$ be the set of limit points of the orbit
so that
$$\o (x) \ = \ \bigcap_{N \in \N} \ \overline{ \{ f^n(x) : n \geq N \}}, \quad \mbox{and}
\quad \O(x) \cup \omega (x) = \overline{\O(x)},$$
the orbit closure of $x$.

A point $x$ is called \emph{recurrent} when $x \in \o (x)$.

We denote by
$$\O^-(x) = \bigcup \limits_{n \in \N} \{ f^{-n}(x)   \} \ = \ \{ y \in X : f^n(y) = x \ \mbox{for some} \ n \in \N \}.$$
and call it the \emph{negative orbit} of $x$.  We also  define the partial orbit sets.  For $N \in \N$
$$\O_N(x) \ = \ \{f^n (x) : 1 \leq n \leq N \}, \ \mbox{and} \ \O_N^-(x) = \bigcup \limits_{ 1 \leq n \leq N} \{ f^{-n}(x)  \}.$$

A subset $A \subset X$ is called \emph{$ +$ invariant} if $f(A) \subset A$ (or, equivalently, if $A \subset f^{-1}(A)$),
\emph{$-$ invariant} if $f^{-1}(A) \subset A$
and \emph{invariant} if $f(A) = A$. For example, $\O(x)$ is $+ $ invariant, $\O^-(x)$ is $- $ invariant.
Clearly, $A$ is $+ $ invariant iff $x \in A $ implies $\O(x)  \subset A$ and
$A$ is $- $ invariant iff $x \in A$ implies $\O^{-}(x) \subset A$.

For every $x \in X$ the set of limit points $\o (x) $ is invariant. Notice that $y \in \o (x)$ if and only if it
is the limit of a subsequence of the orbit sequence $\{ f^n(x) : n \in \N \}$. It easily follows that
\begin{equation}\label{omega}
f(\o (x)) \ = \ \o (x) \ = \ \o (f(x)).
\end{equation}

Clearly, a set is $+ $ invariant if and only if its complement is $- $ invariant.

We call a set $A \subset X$ to be \emph{weakly $- $ invariant} if $A \subset f(A)$ or, equivalently,
if for all $x \in A$ there exists $x_1 \in A$
such that $f(x_1) = x$. Thus, $A$ is weakly $- $ invariant if and only if $f(A \cap f^{-1}(A)) = A$.

\begin{lem}\label{lem1.01a} Let $A \subset X$.

(a) If $A$ is $- $ invariant and $f$ is surjective, then $A$ is weakly $- $ invariant.

(b) $A$ is invariant, i. e. $A = f(A)$, if and only if $A$ is both $+ $ invariant and weakly $- $ invariant.
$A = f^{-1}(A)$ if and only
if $A$ is both $+$ invariant and $-$ invariant and, in that case, it is invariant.

(c) If $A$ is either $+ $ invariant, weakly $- $ invariant or invariant then the
closure $\ol{A}$ satisfies the corresponding property.

(d) Assume $A$ is closed and nonempty. If $A$ is either $+ $ invariant or weakly $- $ invariant then $A$ contains a
closed, nonempty invariant set.  If $A$ is $-$ invariant then it contains a closed, nonempty subset which is
$- $ invariant and invariant.

(e) Assume $f$ is an open map. If $A$ is $+ $ invariant then the interior $A^{\circ}$ is $+ $ invariant.
If $A$ is $- $ invariant then $\ol{A}$ is $- $ invariant.
\end{lem}

{\bfseries Proof:} (a) If $x \in A \subset f(X)$ then there exists $x_1 \in X$ such that $f(x_1) = x$. Because $A$ is
$- $ invariant, $x_1 \in A$.

(b) These are obvious from the definitions since $f$ is surjective.

(c) By compactness, $f(\ol{A}) = \ol{f(A)}$. The results follow from the monotonicity of the closure operator.

(d) If $A$ is $+ $ invariant, let $A_0 = A$ and $A_n = f^n(A)$ for $n \in \N$. This is a decreasing sequence of nonempty closed
sets with $f(A_n) = A_{n+1}$. Let $B \ = \ \bigcap_n \ A_n$, a nonempty closed set by compactness. For all $n$
$f(B) \subset f(A_n) = A_{n+1}$. So $B$ is $+ $ invariant. If $x \in B$ then $x \in A_{n+1}$ implies $f^{-1}(x) \cap A_n \not=
\emptyset$. Hence, $\bigcap_n \ f^{-1}(x) \cap A_n  \ = \ f^{-1}(x) \cap B$ is nonempty by compactness again. Hence, $B$ is
$- $ invariant and so is invariant by (b).

If $A$ is weakly $- $ invariant, let $A_0 = A$ and, inductively, let $A_{n} = A_{n-1} \cap f^{-1}(A_{n-1})$. This is a
decreasing sequence of nonempty closed sets with $f(A_n) = A_{n - 1}$. Let $B = \bigcap_n \ A_n$ and proceed as before.

If $A$ is $- $ invariant, let $A_0 = A$ and, inductively, $A_{n} = f^{-1}(A_{n-1})$. Again, this is a
decreasing sequence of nonempty closed sets with $f(A_n) = A_{n - 1}$. Let $B = \bigcap_n \ A_n$ as before and observe that
 $f^{-1}(B) = \bigcap_n \ f^{-1}(A_n) = \bigcap_n A_{n+1} = B$.

(e) If $f$ is open then $f(A^{\circ})$ is an open subset of $f(A)$.  If also $A$ is $+ $ invariant then
$$f(A^{\circ}) \ \subset \ f(A)^{\circ} \ \subset \ A^{\circ}.$$
The $- $ invariance result follows by taking complements.

$\Box$ \vspace{.5cm}

{\bfseries Remark:} (a) The constructions in part (d) yield the maximum invariant subset in each case.

(b)The closure of a $- $ invariant set need not, in general, be $- $ invariant.
Let $X = \{ 0 \} \cup \{ \frac{1}{n} : n \in \N \}$.
Define $f$ on $X$ by $0, 1 \mapsto 0$ and $\frac{1}{n+1} \mapsto  \frac{1}{n}$ for $n \in \N$. The set
$A = \{ \frac{1}{n+1} : n \in \N \}$ is $- $ invariant but its closure is not.
\vspace{.5cm}

A subset $A \subset X$ is called \emph{$\ep$ dense} if it meets every open $\ep$ ball in $X$.

\begin{lem}\label{lem1.01} Let  $\{ A_n \}$ be a sequence of subsets of $X$.

(a) $\bigcup_{n = 1}^{\infty} \ A_n$ is dense if and only if for every $\ep > 0$ there exists $N \in \N$ such that
$\bigcup_{ n = 1}^{N} \ A_n$ is $\ep$ dense.

(b) If each $A_n$ is open and $\bigcup_{n = 1}^{\infty} \ A_n = X$  then there exists $N \in \N$ such that
$\bigcup_{ n = 1}^{N} \ A_n = X$.
\end{lem}

{\bfseries Proof:} (a) Because $X$ is compact it has a
finite cover by $\ep/2$ balls. A set which meets each of these is $\ep$ dense.

(b) This follows from compactness.

$\Box$ \vspace{.5cm}

\begin{df}\label{df1.01a} A system $(X,f)$ is called \emph{exact} if for every pair of opene subsets $U,V \subset X$
there exists $n \in \N$ such  that $f^n(U) \cap f^n(V) \ \not= \emptyset$.

We call the system \emph{fully exact} if for every pair of opene subsets $U,V \subset X$
there exists $n \in \N$ such  that $(f^n(U) \cap f^n(V))^{\circ} \ \not= \emptyset$.
\end{df}

\vspace{.5cm}

The definition of exactness follows \cite{kss}.  Thus, $(X,f)$ is exact (or fully exact) if and only if
for  every pair of opene subsets $U,V \subset X$ we have  $\bigcup_n \ (f^n(U) \cap f^n(V))  \ \not= \emptyset$
(resp.  $\bigcup_n \ (f^n(U) \cap f^n(V))^{\circ}  \ \not= \emptyset$).

Clearly, if $f$ is an open map then exactness and fully exactness are equivalent.

\begin{theo}\label{theo1.01b} (a) If $(X,f)$ is exact and $f$ is injective then $X$ is a singleton, i. e.
the system is trivial.

(b) $(X,f)$ is fully exact if and only if for
every pair of opene subsets $U,V \subset X$ we have $(\bigcup_n \ f^n(U) \cap f^n(V))^{\circ}  \ \not= \emptyset$.
\end{theo}

\begin{proof} (a) : If $X$ is not a singleton then it contains a pair of disjoint opene sets $U,V$.
Since $f^n(U) \cap f^n(V) \not= \emptyset$ for some $n$, the map $f$ is not injective.

(b) : If the system is fully exact then
$$ \bigcup_n \ (f^n(U) \cap f^n(V))^{\circ} \ \subset \ (\bigcup_n \ f^n(U) \cap f^n(V))^{\circ}$$
implies that the latter set is nonempty.

Let $A \subset U, B\subset V$ be closed sets with nonempty interior. If
the open set $G = (\bigcup_n \ f^n(A) \cap f^n(B))^{\circ} $ is nonempty then it is a Baire space with
a countable, relatively closed cover $\{ G \cap f^n(A) \cap f^n(B) : n \in \N \}$.
So by the Baire Category Theorem some $f^n(A) \cap f^n(B)$ has
nonempty interior in $G$ and so in $X$. In fact, $\bigcup_n \ (f^n(A) \cap f^n(B))^{\circ}$ is dense in $G$.
It follows that $(X,f)$ is fully exact. \end{proof}

$\Box$ \vspace{.5cm}

We note that these two notions are different and in general far away from the property of transitivity.

\begin{ex} There exist systems which are exact but not fully
exact and which are fully exact and surjective but not topologically transitive \end{ex} Consider the system $([-1,1],f)$ with
$$ f(x) = \left\{
            \begin{array}{ll}
              -(2+2x), & \hbox{$-1 \leq x \leq -1/2$;} \\
              2x, & \hbox{$-1/2 \leq x \leq 1/2$;} \\
              2 - 2x, & \hbox{$1/2 \leq x \leq 1$.}
            \end{array}
          \right.$$

Then any interval contained in $[-1,0]$ or $[0,1]$ eventually covers $[-1,0]$ or $[0,1]$ respectively with intersection $\{0\}$. This system is exact but not fully exact.

Whereas consider the system $([-1,1],g)$ with
$$ g(x) = \left\{
            \begin{array}{ll}
              -(2+2x), & \hbox{$-1 \leq x \leq -1/2$;} \\
              2x, & \hbox{$-1/2 \leq x \leq 1/2$;} \\
              3 - 4x  , & \hbox{$1/2 \leq x \leq 1$.}
            \end{array}
          \right.$$

Then any interval contained in $[-1,0]$ or $[0,1]$ eventually covers $[-1,0]$ or $[-1,1]$ respectively with intersection  $[-1,0]$. This system is fully exact.

The reader can refer to Theorem \ref{theo1.03} below to see that these systems are not transitive.

$\Box$ \vspace{.5cm}

\begin{df}\label{def1.02} $(X,f)$ is called:
\begin{itemize}
\item[1.] \emph{Topologically Transitive (TT)} if for every opene $U \subset X$, $\bigcup \limits_{n=1}^\infty \ f^n(U)$ is
dense in $X$.

\item[2.]  \emph{Strongly Transitive (ST)} if for every opene $U \subset X$, $\bigcup \limits_{n=1}^\infty \ f^n(U) = X$.

\item[3.] \emph{Very Strongly Transitive (VST)} if for every opene $U \subset X$ there is a
 $N \in \N$ such that $\bigcup \limits_{n=1}^N \ f^n(U) = X$.

\item[4.] \emph{Minimal (M)} if there is no proper, nonempty, closed invariant subset of $X$.

\item[5.] \emph{Weak Mixing (WM)} if the product system $(X \times X, f \times f)$ is topologically transitive.

\item[6.] \emph{Exact Transitive (ET)}  if for every pair of opene sets $U,V \subset X$, \\
$\bigcup \limits_{n=1}^\infty \ (f^n(U) \cap f^n(V))$ is dense in $X$.

\item[7.] \emph{Strongly Exact Transitive (SET)}  if for every pair of opene sets $U,V \subset X$,
$\bigcup \limits_{n=1}^\infty \ (f^n(U) \cap f^n(V)) = X$.

\item[8.] \emph{Strongly Product Transitive (SPT)} if for every positive integer $k$ the product system $(X^k, f^{(k)})$
is strongly transitive.

\item[9.] \emph{Mixing or Topologically Mixing (TM) } if for every pair of opene  sets $U, V \subset X$,
there exists an $N \in \N$ such that $f^n(U) \cap V \neq \emptyset$ for all $n \geq N$.

\item[10.] \emph{Locally Eventually Onto (LEO)} if for every opene $U \subset X$,
there exists $N \in \N$ such that $f^N(U) = X$, and  so $f^n(U) = X$ for all $n \geq N$.
\end{itemize}
 \end{df}
\vspace{.5cm}

In the case of the property SPT , we rejected the obvious label  \emph{Strongly Weak Mixing}.

We now provide equivalent descriptions of these properties.

For topological transitivity there are many equivalences in the literature, see e. g.  \cite{ ac, avksm, k, sil}.

\begin{theo} \label{theo1.03} For a dynamical system  $(X,f)$ the
following are equivalent.
\begin{itemize}
\item[(1)]  The system is topologically transitive.

\item[(2)] For every pair of  opene sets $U$ and $V$ in $X$, there
exists $n \in \N$ such  that $f^{-n}(U) \cap V \neq
\emptyset$.

\item[(3)]  For every pair of opene
sets $U, V \subset X$ the set $N(U,V)$ is nonempty.

\item[(4)]  For every pair of opene
sets $U, V \subset X$ the set $N(U,V)$ is infinite.

\item[(5)] There exists $x \in X$ such that the orbit $\O(x)$ is dense in $X$,
i. e. the set $Trans(f)$ of transitive points is nonempty.

\item[(6)] The set $Trans(f)$ of transitive points equals $ \{ x : \omega (x) = X \}$ and it
is a dense, $G_{\delta}$ subset of $X$.

\item[(7)] For every  opene set $U \subset X$,
$\bigcup \limits_{n=1}^{\infty} f^n(U)$ is dense in $X$.

\item[(8)] For every  opene set $U \subset X$, and $\ep > 0$, there exists $N \in \N$ such that
$\bigcup \limits_{n=1}^{N} f^n(U)$ is $\ep$ dense in $X$.

\item[(9)] For every  opene set $U \subset X$,
$\bigcup \limits_{n=1}^{\infty} f^{-n}(U)$ is dense in $X$.

\item[(10)] For every  opene set $U \subset X$, and $\ep > 0$, there exists $N \in \N$ such that
$\bigcup \limits_{n=1}^{N} f^{-n}(U)$ is $\ep$ dense in $X$.

\item[(11)] If $U \subset X$ is opene and $-$ invariant, then  $U$ is dense in $X$.

\item[(12)] If $E \subset X$ is closed and $+ $ invariant, then $E = X$
or $E$ is nowhere dense in $X$.

\end{itemize}

If $(X,f)$ is topologically transitive, then  $f$ is surjective and either $X$ is a single periodic orbit or
it is a perfect space, i.e. has no isolated points.
\end{theo}

{\bfseries Proof:} Notice that condition (7) is the definition of TT and so, of course, (1) $\Leftrightarrow$ (7).
Condition (7) clearly implies that $f(X)$ is dense and so equals $X$ by compactness, i.e. $f$ is surjective.

We first show that if $Trans(f) \not= \emptyset$ then  either $X$ is a periodic
orbit or it is perfect.  Furthermore, $x \in Trans(f)$ implies $\o (x) = X$.

Assume $\O(x) $ is dense. If $x \in \O(x)$ then $x$ is a periodic point
with finite orbit $\O(x)$ dense in $X$ and equals $\o (x)$.
That is, $X$ is a periodic
orbit.  If $x \not\in \O(x)$ then the dense set $\O(x)$ is not closed and so is infinite. Thus, the points of the orbit
are all distinct. This implies that if $y \in X \setminus \O(x)$ then it is the limit of some sequence $f^{n_i}(x)$ with
$n_i \in \N$ and $n_i \to \infty$. In particular, there is such a sequence with $f^{n_i}(x) \to x$ and so
$f^{n_i + k}(x) \to f^k(x)$ for all $k \in \N$. Thus, no point of $X$ is isolated and every point is contained in
$\o (x)$.

Each of   (2), (7), (9) $\Leftrightarrow$ (3) is an easy exercise. (9) $\Leftrightarrow$ (11) since
$\bigcup \limits_{n=1}^{\infty} f^{-n}(U)$ is $-$ invariant and equals $U$ if $U$ is $-$ invariant.
(7) $\Leftrightarrow$ (8) and (9) $\Leftrightarrow$ (10) from Lemma \ref{lem1.01} (a). (11) $\Leftrightarrow$ (12) by taking
complements.

This leaves (4), (5) and (6).

(6) $\Rightarrow$ (5), and (4) $\Rightarrow$ (3)  are obvious.

(9) $\Rightarrow$ (6):

It is clear that
 $$ Trans(f) \ = \ \bigcap_U \ \bigcup \limits_{n=1}^{\infty} f^{-n}(U),$$
with $U$ varying over a countable base. By assumption (9) each $\bigcup \limits_{n=1}^{\infty} f^{-n}(U)$ is
a dense open set. The  Baire Category Theorem then implies that $Trans(f)$ is a dense $G_{\d}$ set. By our initial argument,
$Trans(f) = \{ x : \o (x) = X \}$.

(5) $\Rightarrow$ (4): (4) is obvious if $X$ is a periodic orbit. Otherwise, (5) and our initial argument
imply that $X$ is perfect. If $\O(x)$ is dense then it meets every opene set in an infinite set because $X$ is perfect.
It then follows  that $N(U,V)$ is infinite for every opene pair $U, V$.

$\Box$ \vspace{.5cm}

{\bfseries Remark:} Clearly, (12) implies that if $(X,f)$ is topologically transitive then
$X$ is not the union of two proper, closed, $+ $invariant subsets. If $X$ is perfect then the converse is true as well,
see \cite{ac}.
\vspace{.5cm}

\begin{cor}\label{cor1.03a} For a system $(X,f)$ the set $Trans(f)$ is invariant and $ - $ invariant.

Every transitive point is recurrent.\end{cor}

{\bfseries Proof:} Since $Trans(f) = \{ x : \o (x) = X \} $ by (6), and $\o (x) = \o (f(x))$ by \ref{omega},
it follows that the set of transitive points is invariant and $ - $ invariant. In particular,
$x \in \o (x)$ says that every transitive point is
recurrent.

$\Box$
\vspace{.5cm}

\begin{prop}\label{prop1.03b} Let $(X,f)$ be a topologically transitive system with $X$ infinite. For any $x \in X$ and $n \in \N$
$f^{-n}(x)$ is nowhere dense and so $O^-(x)$ is of first category. \end{prop}

{\bfseries Proof:} Suppose $U \subset f^{-n}(x)$ is opene.  There exists $y \in U \cap Trans(f)$. Since
$Trans(f)$ is  invariant, $x \in Trans(f)$.  Hence, $f^k(x) \in U$ for some $k \in \N$. This implies that
$f^{k+n}(x) = x$ and so $x$ is a periodic point which is also a transitive point.  It follows that $X$ consists exactly
of the orbit of $x$ and so it is finite. Contrapositively, if $X$ is infinite then the closed set $f^{-n}(x)$ is
nowhere dense for all $n \in \N$. As the countable union of nowhere dense sets, $O^-(x)$ is of first category.

$\Box$ \vspace{.5cm}

Recall our identification of a singleton with the point it contains.  In particular,
$$ N(U,x) \quad = \quad \{ n \in \N : x \in f^n(U) \} \ = \ \{ n \in \N : f^{-n}(x) \cap U \not= \emptyset \}.$$

\begin{theo} \label{theo1.04} For a dynamical system $(X,f)$  the
following are equivalent.
\begin{itemize}
\item[(1)]  The system is strongly  transitive.

\item[(2)] For every  opene set $U \subset X$ and every point $x \in X$, there
exists $n \in \N$ such  that $x \in f^n(U)$.

\item[(3)]  For every  opene set $U \subset X$ and every point $x \in X$,
 the set $N(U,x)$ is nonempty.

\item[(4)]  For every  opene set $U \subset X$ and every point $x \in X$,
 the set $N(U,x)$ is infinite.

\item[(5)] For every $x \in X$,  the negative orbit $\O^-(x)$ is dense in $X$.

\item[(6)]  For every $x \in X$,  and $\ep > 0$, there exists $N \in \N$ such that
$\O^-_N(x)$ is $\ep$ dense in $X$.

\item[(7)] If $E \subset X$ is nonempty and $- $ invariant, then  $E$ is  dense in $X$.

\end{itemize}

If $(X,f)$ is strongly  transitive, then $f$ is topologically transitive.
\end{theo}

{\bfseries Proof:} That each of (1), (3) and (5) $\Leftrightarrow$ (2) are easy exercises.
(5) $\Leftrightarrow$ (6) by Lemma \ref{lem1.01} (a).

(4) $\Rightarrow$ (3) : Obvious.

(5) $\Rightarrow$ (4) : If $n \in N(U,x)$ then there exists $y \in U$ with $f^n(y) = x$. Because $\O^-(y)$ is dense
there exists $k \in N(U,y)$. That is, there exists $z \in U$ such that $f^k(z) = y$ and so $f^{k+n}(z) = x$. Hence,
$k + n \in N(U,x)$.  Thus, the set $N(U,x)$ is unbounded.

(5) $\Leftrightarrow$ (7):  $\O^-(x)$ is $- $ invariant and if $x \in E$ and $E$ is $- $ invariant then
$\O^-(x) \subset E$.

Condition (3) implies condition (3) of Theorem \ref{theo1.03}. Hence, strongly transitive implies topological transitive.

$\Box$ \vspace{.5cm}

Recall that a subset $L \subset \N$ is \emph{syndetic} if there exists $N \in \N$ such that every interval of length
$N$ in $\N$ meets $L$.

\begin{theo}\label{theo1.04a} For a dynamical system  $(X,f)$  the
following are equivalent.
\begin{itemize}
\item[(1)]  The system is very strongly transitive.

\item[(2)]  For all $\ep > 0$, there exists $N \in \N$ such that
$\O^-_N(x)$ is $\ep$ dense in $X$ for every $x \in X$.

\end{itemize}

If $(X,f)$ is very strongly transitive then
for every  opene set $U \subset X$ and every point $x \in X$,
 the set $N(U,x)$ is syndetic. \end{theo}

 {\bfseries Proof:} (1) $\Rightarrow$ (2):
 Cover $X$ by $\epsilon/2$ balls $V_1, \ldots , V_m$.
There exists an $N \in \N$  large enough that
 $\bigcup_{n=1}^N \ f^n(V_i) = X$, for $i = 1, \ldots, m$.
 Fix an
 $x \in X$, then for any $y \in X$,  $y \in V_i$ for
 some $i$ and $x \in f^n(V_i)$ for some $n$ with $1 \leq n \leq N$. Therefore $\O_N^- (x)$ is $\epsilon$ dense.

(2) $\Rightarrow$ (1): Suppose that
$N \in \N$ is such that $\O^-_N (x)$ is $\epsilon$ dense for
every $x \in X$. Let $W$ be an $\epsilon$ ball  in $X$.
For $x \in X$  there exists $ x' \in W$ with $f^n(x') = x$
where $1 \leq n \leq N$. Hence, $x \in \bigcup_{n=1}^N   f^n(W)$.
Since $x$ is arbitrary, $\bigcup_{n=1}^N  \  f^n(W) = X$.

 If $X = \bigcup_{n = 1}^{N} \ f^n(U)$ then for every
 $k \in \N, \ X = f^k(X) = \bigcup_{n = k+1}^{N+k} \ f^n(U)$.
Thus, for every $x \in X$, the set $N(U,x)$ meets every interval of length $N$ in $\N$.

$\Box$ \vspace{.5cm}

\begin{cor}\label{cor1.04b} If $(X,f)$ is very strongly transitive then for any opene $U, V \subset X$,
the set $N(U,V)$ is syndetic. \end{cor}

{\bfseries Proof:} If $x \in V$, then $N(U,x) \subset N(U,V)$.

$\Box$ \vspace{.5cm}

\begin{theo}\label{theo1.05} If $f$ is an open map then the following are equivalent.
\begin{itemize}
\item[(1)]  The system is very strongly transitive.

\item[(2)]  The system is strongly transitive.

\item[(3)]  $X$ does not contain a proper, closed $- $ invariant subset.
\end{itemize}

\end{theo}

{\bfseries Proof:} (1) $\Rightarrow$ (2) $\Rightarrow$ (3) whether the map is open or not.

(2) $\Rightarrow$ (1): If $U$ is opene then each $f^n(U)$ is open and so if $\{ f^n(U) : n \in \N \}$ covers $X$ then
it has a finite subcover.

(3) $\Rightarrow$ (2):  If $E$ is a nonempty $- $ invariant subset of $X$ then by Lemma \ref{lem1.01a} (e)
$\ol{E}$ is a nonempty, closed $- $ invariant subset and so it equals $X$.  Thus, $E$ is dense.  This verifies
condition (6) of Theorem \ref{theo1.04}.

$\Box$ \vspace{.5cm}

\begin{theo} \label{theo1.06} For a dynamical system $(X,f)$  the
following are equivalent.
\begin{itemize}
\item[(1)]  The system is minimal.

\item[(2)] For every  opene set $U \subset X$ and every point $x \in X$, there
exists $n \in \N$ such  that $f^n(x) \in U$.

\item[(3)]  For every  opene set $U \subset X$ and every point $x \in X$,
 the set $N(x,U)$ is nonempty.

 \item[(4)]  For every  opene set $U \subset X$ and every point $x \in X$,
 the set $N(x,U)$ is syndetic.

\item[(5)] For every $x \in X$,  the  orbit $\O(x)$ is dense in $X$.

\item[(6)] The set $Trans(f)$ is equal to the entire space $X$. 

\item[(7)] For every $x \in X, \ \o f(x) = X$.

\item[(8)]   For every  opene set $U \subset X$, $\bigcup_{n = 1}^{\infty} \ f^{-n}(U) \ = \ X$.

\item[(9)]   For every  opene set $U \subset X$, there exists $N \in \N$ such that
 $\bigcup_{n = 1}^{N} \ f^{-n}(U) \ = \ X$.

\item[(10)] If $E \subset X$ is nonempty, closed and invariant, then  $E = X$.

\item[(11)] If $E \subset X$ is nonempty, closed and $ + $ invariant, then  $E = X$.

\item[(12)] If $E \subset X$ is nonempty, closed and weakly $ - $ invariant, then  $E = X$.
\end{itemize}

If $(X,f)$ is minimal, then it is very strongly transitive.
\end{theo}

{\bfseries Proof:} Notice that condition (8) is the definition of minimality and so, of course, (1) $\Leftrightarrow$ (8).

Again each of the equivalences   (2), (3) and (8) $\Leftrightarrow$ (5) is an easy exercise.

(5)  $\Leftrightarrow$ (6) is obvious and (6) $\Leftrightarrow$ (7) by Theorem \ref{theo1.03} (6).

(8) $\Leftrightarrow$ (9) by compactness.

(4) $\Rightarrow$ (3) : Obvious.

(9) $\Rightarrow$ (4) :  If $X = \bigcup_{n = 1}^{N} \ f^{-n}(U)$ then for every
 $k \in \N, \ X = f^{-k}(X) = \bigcup_{n = k+1}^{N+k} \ f^{-n}(U)$.
Thus, for every $x \in X$, the set $N(x,U)$ meets every interval of length $N$ in $\N$.

(5) $\Leftrightarrow$ (11): $\ol{\O(x)}$ is closed and $+ $ invariant and if $x \in E$ with $E$ $+ $ invariant then
$\O(x) \subset E$.

(11), (12) $\Rightarrow$ (10): Obvious.

(10) $\Rightarrow$ (11) and (12):  If $E$ is nonempty, closed and either $+ $ invariant or weakly $- $ invariant then
by Lemma \ref{lem1.01a} (d) $E$ contains a nonempty, closed invariant subset.

If $(X,f)$ is minimal and $U$ is opene then by (9) there  exists $N \in \N$ such that
 $\bigcup_{n = 1}^{N} \ f^{-n}(U) \ = \ X$. Since $Trans(f) = X$ the system is topologically transitive and so by
Theorem \ref{theo1.03} $f$ is surjective. Apply $f^{N+1}$ to get that  $\bigcup_{n = 1}^{N} \ f^{n}(U) \ = \ X$.
That is, $f$ is very strongly transitive.

$\Box$ \vspace{.5cm}

\begin{theo}\label{theo1.07} If $f$ is a homeomorphism, $(X,f)$ is strongly transitive
if and only if it is very strongly transitive if and only if
it is minimal.
\end{theo}

{\bfseries Proof:} A homeomorphism is open and so strongly
transitive $\Leftrightarrow$ very strongly transitive by Theorem \ref{theo1.05}.
In any case, minimality $\Rightarrow$ very strongly
transitive by Theorem \ref{theo1.06}.  For a homeomorphism we can reverse the argument.

For very strongly transitive $\Rightarrow$ minimality we see that   $\bigcup_{n = 1}^{N} \ f^{n}(U) \ = \ X$, apply $f^{-N-1}$ to get
 $\bigcup_{n = 1}^{N} \ f^{-n}(U) \ = \ X$.

 $\Box$ \vspace{.5cm}

Let $h : X_1 \to X_2$ be a continuous surjection between compact spaces.  $h$ is called \emph{irreducible}
if $A \subset X_1$ with $A$ closed and $h(A) = X_2$ implies $A = X_1$.  We define
$$Inj_h \quad = \quad \{ x \in X_1 : \{ x \} = h^{-1}(h(x)) \}.$$
We call $h$ \emph{almost one-to-one} when $Inj_h$ is dense in $X_1$. An almost one-to-one map is clearly
irreducible and if the spaces are metrizable then the converse holds. Furthermore, if $U \subset X_1$ is
open and $x \in U \cap Inj_h$ then $h(U)$ is a neighborhood of $h(x)$.  For details, see, for example,
\cite{ag} Lemma 1.1.

We recover a result from \cite{k2}

\begin{cor}\label{cor1.08} If $(X,f)$ is minimal then $f : X \to X$ is an almost one-to-one map. In particular,
$(X,f)$ is not fully exact unless it is trivial. \end{cor}

{\bfseries Proof:} Since the system is minimal, $f$ is surjective.  It suffices to show that $f$ is irreducible.

If $A \subset X$ is closed and $f(A) = X$ then $A \subset f(A)$ and so $A$ is a closed, nonempty, weakly
$- $ invariant subset.  Hence, $A = X$ by (9) of Theorem \ref{theo1.06}.

Assume that $(X,f)$ is fully exact and nontrivial. Let $U,V$ be disjoint opene sets and let $n \in \N$ and
$W$ opene such that $f^n(U) \cap f^n(V) \supset W$. Let $U_1 = U \cap f^{-n}(W)$ and $V_1 = V \cap f^{-n}(W)$.
These are disjoint opene sets with $f^n(U_1) = f^n(V_1)$. It follows that $(U_1 \cup V_1) \cap Inj_{f^n} = \emptyset$.
Thus, $f^n$ is not almost one-to-one and so is not irreducible.  Since the composition of irreducible maps is
irreducible, it follows that $f$ is not irreducible. By the above argument, $(X,f)$ is not minimal.

$\Box$ \vspace{.5cm}

\begin{theo} \label{theo1.09} For a dynamical system  $(X,f)$  the
following are equivalent.
\begin{itemize}
\item[(1)]  The system is weak mixing.

\item[(2)] For triple of  opene sets $U, V, W$ in $X$, there
exists $N \in \N$ such  that $f^{-N}(U) \cap V \neq
\emptyset$ and $f^{-N}(U) \cap W \neq \emptyset$.

\item[(3)] For triple of  opene sets $U, V, W$ in $X$, there
exists $N \in \N$ such  that $f^{N}(U) \cap V \neq
\emptyset$ and $f^N(U) \cap W \neq \emptyset$.

\item[(4)]  For every $N \in \N$ the product system $(X^N, f^{(N)})$ is
topologically transitive.

\item[(5)] For every  opene set $U \subset X$, and $\ep > 0$, there exists $N \in \N$ such that
$f^{-N}(U)$ is $\ep$ dense in $X$.

\item[(6)] For every  opene set $U \subset X$, and $\ep > 0$,
$f^{-N}(U)$ is $\ep$ dense in $X$ for infinitely many $N \in \N$.

\item[(7)] For every  opene set $U \subset X$, and $\ep > 0$, there exists $N \in \N$ such that
$f^{N}(U)$ is $\ep$ dense in $X$.

\item[(8)] For every  opene set $U \subset X$, and $\ep > 0$,
$f^{N}(U)$ is $\ep$ dense in $X$ for infinitely many $N \in \N$.

\end{itemize}

\end{theo}

{\bfseries Proof:}  That (1) $\Leftrightarrow$ (2) and (3) are characterizations of weak mixing given by Petersen, \cite{pk}.
The equivalence (1) $\Leftrightarrow$ (4) is a well-known consequence of the Furstenberg Intersection Lemma. Both are reviewed,
for example, in \cite{aan}.

(5) $\Rightarrow$ (2) and(7) $\Rightarrow$ (3): We choose $\ep > 0$ small enough that both $V$ and $W$ contain $\ep$ balls.

(6) $\Rightarrow$ (5) and (8) $\Rightarrow$ (7): Obvious.

(4) $\Rightarrow$ (6) and (8): Let $V_1,...,V_k$ be a finite cover of $X$ by $\ep/2$ balls. Because the product
system $(X^k,f^k)$ is topologically transitive, there exist infinitely many $N_1, N_2$ such that
$N_1 \in N(U,V_1) \cap ... \cap N(U,V_k)$ and $N_2 \in N(V_1,U) \cap ... \cap N(V_k,U)$. these conditions imply that
 $f^{N_1}(U)$ and
$f^{-N_2}(U)$ are $\ep$ dense.

$\Box$ \vspace{.5cm}

\begin{theo}\label{theo1.09aa} Let $(X,f)$ be a dynamical system.

(a) If $(X,f)$ is exact transitive then it is weak mixing.

(b) The following conditions are equivalent.
\begin{itemize}
\item[(1)]  The system is strongly exact transitive.

\item[(2)] For every pair of open sets $U, V$, $\bigcup_{n \in \N} \ (f \times f)^n(U \times V)$
contains the diagonal $id_X$.

\item[(3)] For every $x \in X$, the negative $f \times f$ orbit $\O^-(x,x)$ is dense in $X \times X$.
\end{itemize}

If $(X,f)$ is strongly exact transitive then it is exact transitive and strongly transitive.
\end{theo}

\begin{proof}(a) Condition (2) of Theorem \ref{theo1.09} clearly holds for an exact transitive system.

(b) All three conditions say that for every $x \in X$ and opene $U, V \subset X$ there exists $n \in \N$ such that
$x \in f^n(U)$ and $x \in f^n(V)$.

$\Box$ \end{proof}

\vspace{.5cm}

\begin{theo}\label{theo1.09ab} (a)  If $(X,f)$ is exact transitive then it is topologically transitive and  exact.

(b) If $(X,f)$ is topologically transitive and fully exact then it is exact transitive.

(c) If $(X,f)$ is strongly exact transitive then it is fully exact.
\end{theo}

\begin{proof} (a) :  Obvious.

(b) : Assume $(X,f)$ is topologically transitive and fully exact.
For an opene pair $U,V$  there exists a transitive point $x$ in
the opene set $(\bigcup_n \ f^n(U) \cap f^n(V))^{\circ}$.  So there exists $n$ so that $x \in f^n(U) \cap f^n(V)$.
The orbit  $O(x)$ is then contained in $\bigcup_{k \geq n} f^k(U) \cap f^k(V)$ and so the latter set is dense.

(c) : This follows from Theorem \ref{theo1.01b} (b).

$\Box$ \end{proof}

\vspace{.5cm}

This shows that exact transitivity is a slight strengthening of the conjunction of exactness and topological transitivity.

\begin{theo}\label{theo1.09a} For a dynamical system $(X,f)$  the
following are equivalent.
\begin{itemize}
\item[(1)]  The system is strongly product transitive.

\item[(2)]  For $\ep > 0$ and every finite subset $F \subset X$, there exists $N \in \N$ such that
$f^{-N}(x)$ is $\ep$ dense in $X$ for all $x \in F$.

\item[(3)]  For $\ep > 0$ and every finite subset $F \subset X$, there exist infinitely many $N \in \N$ such that
$f^{-N}(x)$ is $\ep$ dense in $X$ for all $x \in F$.

\item[(4)] The collection of subsets of $\N$: $\{ N(U,x) : x \in X $ and $U$ opene in $X \}$ has the finite intersection
property (or equivalently it generates a filter of subsets of $\N$).
\end{itemize}

If $(X,f)$ is strongly product transitive then it is strong exact transitive.
\end{theo}

{\bfseries Proof:}  (1) $\Rightarrow$ (3):  Let $\{ U_i : i = 1, ..., K_1 \}$ be a be a finite cover of $X$ by opene $\ep/2$ balls.
Given a finite set $F = \{ x_j : j = 1,...,K_2 \} \subset X$. Let $K = K_1 \cdot K_2$,  and
label the points of $X^K$ by
index pairs $ij$. Define the point $x \in X^K$ by  $x_{ij} = x_j$. Let $U_{ij} = U_i$ and
let $U$ be the opene subset of $X^K$ which is the
product of the $U_{ij}$'s.
Because the product system $(X^K,f^{(K)})$ is strongly transitive, Theorem \ref{theo1.04} (4)
implies that there exist $N$ arbitrarily
large such that for each such $N \ U \cap (f^{(K)})^{-N}(x) \not= \emptyset$ That is, there exists
 $z \in U$ with $(f^{(K)})^N(z) = x$. So for each pair $ij$, $z_{ij} \in U_{ij} = U_i$ and
 $f^N(z_{ij}) = x_{ij} = x_j$.
Thus, $f^{-N}(x_j) \cap U_i \not= \emptyset$ for all $i,j$. Hence, each $f^{-N}(x_j)$ is $\ep$ dense because a set which
meets every element of a cover by $\ep/2$ balls is $\ep$ dense.

(3) $\Rightarrow$ (2): Obvious.

(2) $\Rightarrow$ (4): Given $N(U_1,x_1),...,N(U_k,x_k)$ choose $\ep > 0$ so that each $U_i$ contains an $\ep $ ball and
let $F = \{ x_1,...,x_k \}$.  By (2) there exists $N$ such that $f^{-N}(x_i)$ is $\ep$ dense for all $i$. Thus,
$f^{-N}(x_i)$ meets $U_i$ and so $N \in N(U_i,x_i)$ for all $i$.

(4) $\Rightarrow$ (1): Let $x = (x_1,...,x_k) \in X^k$.   Let $U$ contain $U_1 \times ... \times U_k$. Then
$N(U,x) \supset N(U_1,x_1) \cap ... \cap N(U_k,x_k)$.

If the system is strongly product transitive then by Theorem \ref{theo1.04} (5) applied to $(X \times X, f \times f)$,
the negative orbit $O^-(x,x)$ is dense in $X \times X$. By Theorem \ref{theo1.09aa}(b)(3) it follows that $(X,f)$ is
strong exact transitive.

$\Box$ \vspace{.5cm}

\begin{theo} \label{theo1.10} For a dynamical system $(X,f)$  the
following are equivalent.
\begin{itemize}
\item[(1)]  The system is topologically mixing.

\item[(2)]  For every pair of opene
sets $U, V \subset X$ the set $N(U,V)$ is cofinite.

\item[(3)] For every  opene set $U \subset X$, and $\ep > 0$, there exists $N \in \N$ such that
$f^{-n}(U)$ is $\ep$ dense in $X$ for all $n \geq N$.

\item[(4)] For every  opene set $U \subset X$, and $\ep > 0$, there exists $N \in \N$ such that
$f^{n}(U)$ is $\ep$ dense in $X$  for all $n \geq N$.

\end{itemize}

If $(X,f)$ is topologically mixing then it is weak mixing.

\end{theo}

\begin{proof}  That (1) $\Leftrightarrow$ (2) and (3), (4) $\Rightarrow$ (2) are obvious.

(2) $\Rightarrow$ (3) and (4): Let $V_1,...,V_k$ be a finite cover of $X$ by $\ep/2$ balls. As the intersection of
cofinite sets,
$N(U,V_1) \cap ... \cap N(U,V_k)$ and $ N(V_1,U) \cap ... \cap N(V_k,U)$ are cofinite.

$\Box$
\end{proof}

\vspace{.5cm}

\begin{theo} \label{theo1.10a} For a dynamical system $(X,f)$  the
following are equivalent.
\begin{itemize}
\item[(1)]  The system is locally eventually onto.

\item[(2)]  For all $\ep > 0$, there exists $N \in \N$ such that
$f^{-N}(x)$ is $\ep$ dense in $X$ for every $x \in X$.

\item[(3)]  For all $\ep > 0$, there exists $N \in \N$ such that
$f^{-n}(x)$ is $\ep$ dense in $X$ for every $x \in X$ and every $n \geq N$.
\end{itemize}

If $(X,f)$ is locally eventually onto then it is strongly product transitive and topologically mixing.
\end{theo}

{\bfseries Proof:} Notice that if $f^N(U) = X$ then $f^n(U) = X$ for all $n \geq N$.

(1) $\Rightarrow$ (3): Let $\{U_1,...,U_m \}$ be a cover by $\ep/2$ balls. There exists an $N$ such that
$n \geq N$ implies $f^n(U_i) = X$ for $i = 1, ..., m$. Then $f^{-n}(x)$ meets each $U_i$ for all $x \in X$ and $n \geq N$.
So all such $f^{-n}(x)$ are $\ep$ dense.

(3) $\Rightarrow$ (2): Obvious.

(2) $\Rightarrow$ (1): Given an opene $U$, let $\ep > 0$ be such that $U$ contains an $\ep$ ball. If $f^{-N}(x)$ is
$\ep$ dense for all $x$ then $f^{-N}(x)$ meets $U$ for every $x$.  Thus, $f^N(U) = X$.

$\Box$ \vspace{.5cm}

{\bfseries Remark:} It is clear that if $(X,f)$ is locally eventually onto then $N(U,x)$ is co-finite in $\N$ for all opene
$U \subset X$ and all $x \in X$.  It is not clear that the converse holds even if $f$ is an open map.
\vspace{.5cm}

 The various forms of transitivity that we have discussed can be best
 illustrated in the case of symbolic dynamics. Hence, we consider
 when a subshift satisfies these  transitivity conditions.

For a finite set $\A$ with the discrete topology, let $\A^{\N}$ be the space of all one sided infinite  sequences
provided with the product topology, and let $\s$ be the shift map, defined by $\s(x)_n = x_{n+1}$ for all $n \in \N$.

Regarding $\A$ as an alphabet,  a word,  $v$, is a finite sequence consisting of
letters of $\A$, and we write $|v|$ for the length of the word $v$. If $v$ is a word and $w$ is a word or element of
$\A^{\N}$,  we write $vw \in \V$  for the obvious concatenation.

Let $\Omega$ be a closed, shift invariant subset of $\A^{\N}$. The system $(\Omega,\s)$ is called a
\emph{subshift} of $(\A^{\N},\s)$. $\Omega$ is completely determined by its \lq\lq vocabulary" (language) $\V$,
the collection of all finite words which appear in some $x \in \Omega$.

If $x \in \Omega$ and $n>0$, $x_{[1,n]}$ denotes the word consisting of the first $n$ entries of $x$.

 Notice that as $v$
varies over $\V$, the \emph{cylinder sets}
$$[v]\ = \ \{ x \in \Omega : x_{[1,\ |v|]} = v \}\hspace{2cm}$$
comprise a basis of clopen sets for the topology on $\Omega$.

\begin{theo}\label{theo2.06} Let $(\Omega, \s)$ be a subshift with its associated vocabulary $\V$.

\begin{itemize}
\item[(a)] $(\Omega, \s)$ is topologically transitive if and only if for all $v \in \V$ and  all $w \in \V$, there
exists  $a \in \V$ such that $vaw \in \V$.

\item[(b)] $(\Omega, \s)$ is strongly transitive if and only if for all $v \in \V$ and all $x \in \Omega$, there
exists  $a \in \V$ such that $vax \in \Omega$.

\item[(c)] $(\Omega, \s)$ is very strongly transitive if and only if whenever $v \in \V$, there is a finite collection
$\V_v \subset \V$ such that for  every  $x \in \Omega$, $vax \in \Omega$ for some  $a \in \V _v$.

\item[(d)] $(\Omega, \s)$ is minimal  if and only if whenever $v \in \V$ then $v$ occurs in $x$ for all $x \in \Omega$.

\item[(e)] $(\Omega, \s)$ is weak mixing if and only if whenever $v_1, v_2 \in \V$ with $|v_1| = |v_2|$, for all $w \in \V$ there
exist $a_1, a_2 \in \V$ with $|a_1| = |a_2|$ and $v_1a_1w, v_2a_2w \in \V$.

\item[(f)] $(\Omega, \s)$ is  exact  if and only if whenever $v_1, v_2 \in \V$ with $|v_1| = |v_2|$,
 there
exist $x \in \Omega$ and $a_1, a_2 \in \V$ and with $|a_1| = |a_2|$ and   $v_1a_1x, v_2a_2x \in \Omega$.

\item[(g)] $(\Omega, \s)$ is strongly exact transitive if and only if whenever $v_1, v_2 \in \V$ with $|v_1| = |v_2|$,
for all $x \in \Omega$ there
exist $a_1, a_2 \in \V$ with $|a_1| = |a_2|$ and $v_1a_1x, v_2a_2x \in \Omega$.

\item[(h)] $(\Omega, \s)$ is strong product transitive  if and only if whenever $v_1,..., v_n \in \V$ with all of the
same length, and
$x_1,...,x_n \in \Omega$ there exist $a_1,...,a_n$ all of the same length such that $v_1a_1x_1,...,v_na_nx_n \in \Omega$.

 \item[(i)] $(\Omega, \s)$ is topologically mixing if and only if whenever $v, w \in \V$ there
exists $N \in \N$ such that for all $k \in \N$ there exists $a_k \in \V$ with $|a_k| = N + k$ with $va_kw \in \V$.

\item[(j)] $(\Omega, \s)$ is locally eventually onto if and only if whenever $v \in \V$, there is a finite collection
$\V_v \subset \V$ all of whose elements are of the same length such that if
$x \in \Omega$ there is an $a \in \V _v$ with $vax \in \Omega$.
\end{itemize}
\end{theo}

{\bfseries Proof:} (a) : Given $v, w \in \V$ assume there always exists $a$ such that $vaw \in \V$.
Then there exists $x \in \Omega$ which contains this word beginning at
position $i+1$ so $\s^i(x) \in [v]$ and $\s^{i+n}(x) \in [w]$ with $n = |va|$.
Hence, $n \in N([v],[w])$. Thus,  $(\Omega, \s)$ is topologically transitive.

Conversely, suppose $\Omega$ is topologically transitive and $x$ is a transitive point. Let $v,w \in \V$.
Since $\sigma^n(x)$ is also a transitive point for $n>0$ it follows that $v$ and $w$ appear infinitely often in $x$
so certainly a word of the form $vaw$  appears in $\Omega$.

(b) : The given condition says exactly that for each $v \in \V$, every $x \in \Omega$ occurs in some shift of the
cylinder set $[v]$.  Since the $[v]$'s form a basis, this implies strong transitivity.

Conversely, if the system is strongly transitive then $N([v],x)$ is infinite for every $v \in \V$ and $x \in X$.
If $n \in N([v],x)$ with $n > |v|$ then there is a word $a$ of length $n - |v|$ such that $uax \in \Omega$.

(c) : Given any $N \in \N$ there are only finitely many words of length less than $N$. The given condition says
exactly that for each $v \in \V$, every $x \in \Omega$ occurs in some shift of the
cylinder set $[v]$ with an upper bound on the number of shifts required.  This implies very strong transitivity.

Conversely, if the system is very strongly transitive, then given $v \in \V$ there exists $N \in \N$ such that
$\bigcup_{k=1}^N \ \s^k([u]) = \Omega$.  For every $x \in \Omega$, there exists $y \in \Omega$ such that
$\s^{|v|+1}(y) = x$. Then $y \in  \bigcup_{k=1}^N \ \s^k([u])$ and so $x \in \bigcup_{k=|v| + 1}^{N + |v| + 1} \ \s^k([v])$.
This means there exists $a$ with $|a| \leq N$ such that $vax \in \Omega$.

(d) : If $(\Omega, \s)$ is minimal, $v \in \V$ and $x \in \Omega$ then the $\s$ orbit of $x$ enters $[v]$ and so
$v$ occurs in $x$.

Conversely, if every word $v \in \V$ occurs in every $x \in \Omega$ then it occurs in every $\s^k(x)$ and so
occurs in $x$ infinitely often. This implies that $\o (x)$ meets every $[v]$ and so is dense. Hence, the closed set $\o (x)$ equals
$\Omega$ for every $x$.  Thus, $(\Omega, \s)$ is minimal.

(e) : The given condition implies that for every $v_1, v_2, w \in \V$ with $|v_1| = |v_2|$, the
hitting time set $N([v_1],[w]) \cap N([v],[w_2])$ is
nonempty.  Since $\{ [v_1] \times [v_2] : |v_1| = |v_2| \}$ is a basis for $\Omega \times \Omega$,
this implies that the system is weak mixing by the Petersen criterion, see \cite{pk}.

Conversely, if the system is weak mixing then for every $v_1, v_2, w \in \V$ with $|v_1| = |v_2|$, the
hitting time set $N([v_1],[w]) \cap N([v_2],[w])$ is infinite. If $n > |v_1| = |v_2|$ is in the intersection
then there exist $a_1, a_2 \in \V$ with $|a_1| = |a_2| = n - |v|$ such that $v_1a_1w, v_2a_2 \in \V$.

(f), (g) : Proceed as in (e).  We leave the details to the reader.

(h) : Proceed as in (c).  We leave the details to the reader.

(i) : The condition is equivalent to the demand that  $N([v],[w])$ is co-finite for every $v, w \in \V$.

(j) : If the common length of  the elements of $\V_v$ is $n$ then $\s^{|v| + n}([v]) = \Omega$.

Conversely, if the system is locally eventually onto then for all sufficiently large $N$, $\s^N([v]) = \Omega$. If $N > |v|$ then
for every $x \in \Omega$ there exists $a \in \V$ with $|a| = N - |v|$ such that $vax \in \Omega$. Let $\V_v$
be all words of length $N - |v|$.

$\Box$ \vspace{1cm}

Let $(X,f)$ and $ (Y,g)$ be a  dynamical systems.  If
$\pi: X \to Y$ be a continuous surjection such that $\pi \circ f = g \circ \pi$, then $\pi : (X,f) \to (Y,g)$
is called a \emph{factor map}, $(Y,g)$ is called a \emph{factor} of $(X,f)$ and $(X,f)$ is called an \emph{extension} of $(Y,g)$.

In \cite{ag} a property of a dynamical system is called \emph{residual} when it is inherited by factors, by almost one-to-one
lifts and is preserved by inverse limits.  It is  shown there that topological transitivity, minimality, weak mixing and mixing
are residual properties. We now consider  the remaining properties.

\begin{theo}\label{theo2.01} Let $\pi :(X,f) \to (Y,g)$ be a factor map of dynamical systems.

(a) If $(X,f)$ is strongly transitive, very strongly transitive, exact transitive, strongly exact transitive,
exact, strongly product transitive or locally eventually onto then $(Y,g)$ satisfies the corresponding property.

(b) Assume that  $\pi$ is almost one-to-one.  If $(Y,g)$ is very strongly transitive or locally eventually onto  then $(X,f)$
satisfies the corresponding property.

\end{theo}

\begin{proof} (a) Suppose $(X,f)$ is strongly transitive.
Then for any opene $U$ in $Y$, $\pi^{-1}(U)$ is opene in $X$ and so
$X = \bigcup_{n=1}^{\infty} \ f^n(\pi^{-1}(U))$. It follows that
\begin{equation}
\begin{split}
Y \quad = \quad \pi(X) \quad = \quad \bigcup_{n=1}^{\infty} \ \pi f^n(\pi^{-1}(U)) \quad = \\
\bigcup_{n=1}^{\infty} \ g^n \pi (\pi^{-1}(U)) \quad = \quad \bigcup_{n=1}^{\infty} \ g^n(U).\hspace{.5cm}
\end{split}
\end{equation} Thus $(Y,g)$ is strongly transitive.

If $(X,f)$ is very strongly transitive, then we use the same proof with $\bigcup_{n=1}^{\infty} \ f^n$ replaced by
$\bigcup_{n=1}^N \ f^n$ for sufficiently large $N$ depending on $\pi^{-1}(U)$ and obtain the result with
$\bigcup_{n=1}^{\infty} \ g^n$ replaced by $\bigcup_{n=1}^N \ g^n$.  If $(X,f)$ is locally eventually onto, we replace by
$f^N$  and $g^N$.

The remaining properties are similarly proved using $$\pi[f^n(\pi^{-1}(U)) \cap f^n(\pi^{-1}(V))] \subset g^n(U) \cap g^n(V).$$

(b) Now assume that $\pi$ is almost one-to-one and that $U$ is opene in $X$. Let $A$ be a closed set with
nonempty interior. Let $x \in A^{\circ} \cap Inj_{\pi}$.
Since $\pi$ is open at $x$, there exists $V \subset \pi(A^{\circ})$ an open set containing $\pi(x)$.
If $(Y,g)$ is very strongly transitive then $Y = \bigcup_{n=1}^N \ g^n(V)$ for some $N$.
Hence, $Y = \pi( \bigcup_{n=1}^N \ f^n(A))$.  Because $\bigcup_{n=1}^N \ f^n(A)$ is closed and
$\pi$ is irreducible, $X = \bigcup_{n=1}^N \ f^n(A) \subset \bigcup_{n=1}^N \ f^n(U)$. Thus, $(X,f)$ is
very strongly transitive.

If $(X,f)$ is locally eventually onto we replace $ \bigcup_{n=1}^N \ g^n(V)$ by $g^N(V)$ and obtain $X = f^N(U)$ by the same proof.

$\Box$ \end{proof}

\begin{theo}\label{theo2.02} (a) If $(X \times Y, f \times g)$ is
strongly transitive, very strongly transitive, exact transitive, strongly exact transitive,
exact, strongly product transitive or locally eventually onto
then both $(X,f)$ and $(Y,g)$ satisfy the corresponding property.

(b) Assume $(Y,g)$ is mixing. If $(X,f)$ topologically transitive, weak mixing, or mixing then $(X \times Y, f \times g)$
satisfy the corresponding property.

(c) Assume $(Y,g)$ is locally eventually onto.  If $(X,f)$ strongly transitive, very strongly transitive,
exact, full exact, exact transitive, strongly exact transitive, strongly product transitive or locally eventually onto
then $(X \times Y, f \times g)$
satisfy the corresponding property. \end{theo}

{\bfseries Proof:} (a) Each projection is a factor map.

(b) Let $U_1, V_1 \subset X$ and $ U, V \subset Y$ be opene.  Since $(Y,g)$ is mixing $N(U,V)$ is cofinal.
Hence, $N(U_1,V_1) \cap N(U,V)$ is infinite or cofinal if $N(U_1,V_1)$ is. The product result follows for
transitivity and mixing.  Suppose $(X,f)$ is weak mixing. Then $(X \times X, f \times f)$ is topologically
transitive and $(Y \times Y, g \times g)$ is mixing.  Hence, $((X \times Y) \times (X \times Y), (f \times g) \times (f \times g))$
is topologically transitive. Thus, $(X \times Y, f \times g)$ is weak mixing.

(c) Let $U_1 \subset X$ and $U \subset Y$ be opene.  There exist $N$ so that $n \geq N$ implies $g^n(U) = Y$. If
$L \subset \N$ so that $\bigcup \{ f^n(U_1) : n \in L \} = X$ and there exists $n \in L$ with $n \geq N$ then
$\bigcup \{ (f \times g)^n(U_1 \times U) : n \in L \} = X \times Y$.  Using $L = \N$ we obtain the result for
strong topological transitivity.  Using $L = \{ 1,..., N_1 \}$ with $N_1$ sufficiently large, we obtain the result for
very strong topological transitivity. Using $L = \{ N_1 \}$ with $N_1$ sufficiently large, we obtain the result for locally eventually onto.

We leave the others as (easy) exercises.

 $\Box$ \vspace{.5cm}

{\bfseries Remark:} We note that the products of strongly transitive or very strongly transitive systems need not be so.
If $(X,f)$ is nontrivial and minimal with $f$ a homeomorphism then it is very strongly transitive.
The product $(X \times X, f \times f)$ is not minimal and so by Theorem \ref{theo1.07} it is not
strongly transitive.
\vspace{.5cm}

We do not know whether strong transitivity is preserved by almost one-to-one lifts. On the other hand, most
are not necessarily preserved by inverse limits.

Recall that for a surjective system $(X,f)$ the \emph{natural homeomorphism lift} $(\hat X, \hat f)$
is obtained by taking the inverse limit of the inverse sequence of
systems $\{ p_n : (X_{n+1},f_n) \to (X_n,f_n) : n \in \N \}$ with each $(X_n,f_n)$ a copy of $(X,f)$ and
each $p_n = f$. So $\hat x \in \hat X \ \subset \ X^{\N}$ with $\hat x_n = f(\hat x_{n+1})$ for all $n \in \N$.
The map $\hat f$ is the restriction of the product map $f^{\N}$ on $X^{\N}$ to the closed invariant set $\hat X$.
Thus, $\hat f$ is a homeomorphism with inverse the restriction of the shift map to $\hat X$. Each projection
$\pi_n : (\hat X, \hat f) \to (X, f)$ is a factor map.

Now let $(X,f)$ be a nontrivial locally eventually onto system. So $(X,f)$ is full exact and so by
Corollary \ref{cor1.08} it is not minimal.
The inverse limit system $(\hat X, \hat f)$ is not exact since $\hat f$ is a nontrivial homeomorphism.
The system $(\hat X, \hat f)$ is not minimal since its factor $(X,f)$ is not. So by Theorem \ref{theo1.07}
it is not strongly transitive.

An $f: X \to X$ is called \emph{almost open} if for every opene
$U \subset X$, the interior $(f(U))^{\circ}$ is non-empty. An almost one-to-one map is almost open and so
Corollary \ref{cor1.08} implies that minimal maps are always almost open.

\begin{ex}    A locally eventually onto map need not be almost open.\end{ex}Let $g$ be any continuous, surjective map on a compact perfect
metric space $Z$. Let $F = g^{\N}$ on $X = Z^{\N}$. That is,
$F(x)_i = g(x_i)$ for all $i \in \N$.  Notice that $F$ is a continuous surjection on $X$ and $F$ commutes with the
shift map $\sigma$ on $X$, since $(\sigma \circ F)(x)_i = g(x_{i+1}) = (F \circ \sigma)(x)_i$. Let
$f = \sigma \circ F = F \circ \sigma$. Then $(X,f)$ is locally eventually onto, as for any cylinder
$[z] \subset X$, with $|z| = k$, we have $f^k([z]) = X$.

Now choose $g$ so that it is not almost open, e.g. let $Z = [0,1]$ and let $g(t) = 2t$ for $t \in [0,\frac{1}{2}]$ and
$g(t) = 1$ for $t \in [\frac{1}{2},1]$.

For any such $g$ there exists a closed set $K$ with nonempty interior $U$ such that
$g(K)$ is nowhere dense. In the above example,
$K = [\frac{1}{2},1]$ will work. Consider in $X$ the set $\tilde K$ of $x$ such
that $x_2 \in K$. Observe that $\{ x : x_2 \in U \}$ is open. $f(\tilde K) = \{ x : x_1 \in g(K) \}$.
Since $g(K)$ is nowhere dense in $Z$, $f(\tilde K)$ is nowhere dense in $X$.

 $\Box$ \vspace{.5cm}

Let $(X,f)$ be any system, then $f$ is called \emph{iteratively almost open} if for every opene $U \subset X$
$f^n(U)^{\circ} \neq \emptyset$ for infinitely many $n \in \mathbb{N}$.

\begin{prop}\label{prop2.03} If $(X,f)$ is strongly transitive, then $f$ is iteratively almost open. \end{prop}

\begin{proof} Given opene $U$ choose a closed $A \subset U$ with a nonempty interior. Since
 $X = \bigcup_{n=1}^\infty f^n(A)$, the Baire Category Theorem implies that
 for some $n \in \N$, $f^{n}(A)$ contains a closed set $B$ with a nonempty interior.
Similarly, for some $k \in \N$
$f^k(B) \subset f^{i+k}(A)$ contains a closed set with a nonempty interior. Hence, $f^{n}(U)$ has non
empty interior for infinitely many $n \in \N$.

$\Box$ \end{proof}\vspace{.5cm}

 Let $\mathcal M$ be the collection  of closed non-empty negatively invariant subsets
 of $X$ (note that $X \in \mathcal M$). A simple Zorn's Lemma argument shows that any $M \in \mathcal M$
 contains a minimal element of $\M$. We  call the minimal elements of $\M$ \emph{negatively minimal sets}.
 From Lemma \ref{lem1.01a} (d) it follows
 that a negatively minimal set is invariant. Clearly distinct negatively minimal sets are disjoint.

If $x \in X$ we let $M(x)$ denote the intersection of all closed $- $ invariant sets which contain $x$, while it is
a minimum element among these sets it is
not necessarily  negatively minimal invariant as in the previous paragraph.
The sets $M(x)$ need not be disjoint. If $f$ is an open map then $M(x)$ is the closure of $O^-(x)$ by Lemma
\ref{lem1.01a}(e).
If $x' \in M(x)$ then $M(x') \subset M(x)$.

We observe that if $M$ is a negatively minimal set, and $x \in M$, then $M(x)=M$.
Moreover, if whenever $x' \in M(x)$ we have $M(x')=M(x)$ then $M(x)$ is negatively minimal set.

Of course if $(X,f)$ is strongly transitive, then $X$ is negatively minimal. If the map $f$ is open, then the converse holds
by Theorem \ref{theo1.05}.

\medskip

Recall that point $x$ is called \emph{recurrent} when $x \in \o (x)$, i. e. for every open set $U$ containing $x$ there
exist $j \in \N$ such that $f^j(x) \in U$.

\begin{theo}\label{theo2.04} If ($X,f)$ is strongly transitive but not minimal,
then the set of non-recurrent points is dense in $X$.\end{theo}

\begin{proof} Let $M$ be a minimal subset of $X$ and $y \in M$. Since $(X,f)$ is not minimal
$U = X \setminus M$ is opene. Since $O^-(y)$ is dense, it meets $U$.  Let $n$ be the smallest
positive integer such that $f^{-n}(y) \cap U \not= \emptyset$ and let $x$ be a point of this intersection.
Thus, $x \not\in M$, but $f(x) \in f^{-(n-1)}(y) \subset M$. It follows that $\o (x) = \o (f(x)) \subset M$ and so
$x$ is not recurrent. Since $M$ is $+ $ invariant, $O^-(x)$ is disjoint from $M$. On the other hand, for
every point $z \in O^-(x)$, $\o (z) = \o (x) \subset M$ and so the points of the dense set $O^-(x)$ are
non-recurrent.

$\Box$ \end{proof}\vspace{.5cm}

\begin{cor} A strongly transitive $(X,f)$  with all points recurrent is minimal. \end{cor}
$\Box$ \vspace{.5cm}

We say that $(X,f)$ has \emph{dense periodic sets} if for every opene $U$ there exists $N \in \N$ and
$A \subset U$ closed such that $f^N(A) = A$.

\begin{theo}\label{theo2.05} If $(X,f)$ is locally eventually onto then it has dense periodic sets. \end{theo}

{\bfseries Proof:}  Given $U$ opene, let $V \subset U$ be closed with a nonempty interior.  Since the system is locally eventually onto, there exists
$N \in \N$ such that $f^N(V) = X \supset V$. This says that for the system $(X,f^N)$ the set $V$ is weakly $- $ invariant.
By Lemma \ref{lem1.01a} (d) $V$ contains a nonempty closed set $A$ which is invariant for $f^N$.

$\Box$ \vspace{.5cm}

On the other hand, the following example was given to us by Elon Lindenstrauss.

\begin{ex} A locally eventually onto map need not have any periodic points. \end{ex}
Let $(\widehat X,\s)$ be an infinite minimal subshift of the full shift map $(\{ 1,2 \}^{\N},\s)$. Let $(X_0,\s)$ be the
subshift of the full shift map $(\{ 0, 1, 2 \}, \s)$ consisting of all sequences in which the word $00$ does not occur.
Thus, $(X_0,\s)$ is a subshift of finite type and the map $x \mapsto \widehat x$ which excludes all the occurrences of $0$
defines a continuous map $\pi$ from $X_0$ onto $\{ 1,2 \}^{\N}$. Similarly, for any finite word $u$ in the language of $X_0$
we define $\widehat u$ to be the finite word with alphabet $\{ 1,2 \}$ obtained by excluding the $0$'s.
Let $X = \pi^{-1}(\widehat X)$.  While $\pi$ is not an action map, it
is nonetheless clear that $X$ is a $+$ and $-$ invariant subset of $X_0$. For any word $\widehat u$ in the language of $\widehat X$
we obtain words $u$ in the language of $X$ by inserting $0$'s arbitrarily but with no repeats.

Since $\widehat X$ contains no periodic points, is obvious that
that $X$ does not. We now show that $(X,\s)$ is leo. It suffices to show that for any word $u$ in the alphabet
of $X$, there exists $N \in \N$ such that $\s^N([u]) = X$. We may assume that $u$ has length $L$ at least two so that
$\widehat u$ has length $\widehat L > 0$.  Because $(\widehat X,\s)$ is minimal, there exists $M \in \N$ such that
$$\bigcup_{0 \leq k < M} \ \s^{-k}([\widehat u]) = \widehat X.$$

Let $z \in X$ with $\pi(z) = \widehat z \in \widehat X$. Since the shift is surjective on $\widehat X$ we can choose $\widehat w \in \widehat X$
such that $\s^{2M + \widehat L}(\widehat w) = \widehat z$. There exists $k$ with $0 \leq k < M$ such that $\widehat x = \s^{k}(\widehat w) \in [\widehat u]$.
Thus, $\widehat x$ is the concatenation $\widehat u \widehat v \widehat z$ where $\widehat v$ has length between $M+1$ and $2M$. By inserting at most $M$ $0$'s
without repeats between the endpoints, we can define a word $v$ in the language of $X$ with length exactly $2M$. Thus, $uvz \in [u] \subset X$ with
$\s^{L + 2M}(uvz) = z$.  Since $z$ was arbitrary, $\s^{N}([u]) = X$ with $N = L + 2M$.

$\Box$ \vspace{1cm}

 We  have  established

\vskip .3cm

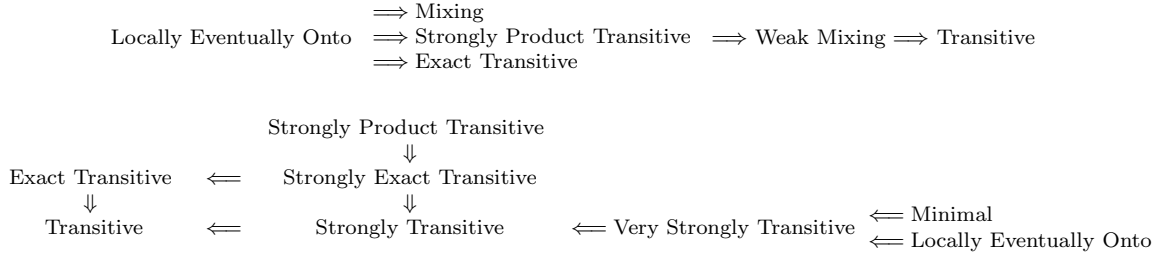
\centerline{\scriptsize{Locally Eventually Onto $\begin{array}{l}
                                \Longrightarrow \text{Mixing} \\
                                \Longrightarrow \text{Strongly Product Transitive}\\
                                    \Longrightarrow \text{Exact Transitive}
                              \end{array}$
  $\Longrightarrow$ Weak Mixing $\Longrightarrow$ Transitive }}

  \vskip .5cm

 \centerline{\scriptsize{    $\begin{array}{ccc}
                                                                              \text{  } & \text{  } & \text{Strongly Product Transitive  }\\
                                                                              \text{  } &  \text{  } & \Downarrow \\
                                                                               \text{  Exact Transitive } & \Longleftarrow  & \text{ Strongly Exact Transitive } \\
                                                                                \Downarrow & \text{  }  & \Downarrow \\
                                                                             \text{ Transitive} & \Longleftarrow  & \text{ Strongly Transitive } \\
                                                                              \text{  } & \text{  } & \text{  }\\
                                                                               \text{  } & \text{  } & \text{  }\\
                                                                                \text{  } & \text{  } & \text{  }\\
                                                                               \text{  } & \text{  } &\text{  }
                                                                             \end{array}
  \Longleftarrow \text{Very Strongly Transitive} \begin{array}{l}
                                                                                                                     \Longleftarrow  \text{Minimal} \\
                                                                                                 \Longleftarrow  \text{Locally Eventually Onto} \end{array}$   }}

The reverse implications do  not hold for most of these.

\begin{prop} The following statements hold:

1. Mixing $\nRightarrow$ Strongly Transitive, Exact Transitive or Minimal.

2. Very Strongly Transitive $\nRightarrow$ Minimal.

3. Minimal $\nRightarrow$ Exact Transitive or Weak Mixing.

4. Strongly Product Transitive \& Mixing $\nRightarrow$ Very Strongly  Transitive.

5. Exact Transitive \& Mixing $\nRightarrow$  Strongly  Transitive.

6. Weak Mixing $\nRightarrow$ Mixing.

\end{prop}

{\bfseries Proof:}  We prove each of these statements by providing examples.

\vskip .3cm

1. \& 2.  The full shift $\sigma$ on $\Omega = \{ 0,1 \}^{\Z}$ is a mixing homeomorphism which is not minimal. Since it
is injective, it is not exact.  By Theorem \ref{theo1.07} it is not ST because it is not minimal.  The full shift on
$\{0, 1 \}^{\N}$ is LEO and so is VST but not minimal.

\vskip .3cm

3. Let $R_\alpha$ be an irrational rotation on the unit circle $\mathbb{S}^1$.
The map $R_\alpha$ is a minimal, isometric homeomorphism.
Hence, it is VST but is not exact or weak mixing.

\vskip .3cm

4. Let $(\Omega,\sigma)$ be subshift on the closure of all sequences of the form
$$0^k{1^{t_1}}0^{3^{n_1}}{1^{t_2}}0^{3^{n_2}}{1^{t_3}}0^{3^{n_3}}
\ldots {1^{t_i}}0^{3^{n_i}} \ldots, \ k \geq 0, t_i, n_i \in \N$$
in  $\{0,1\}^{\N}$.

This is SPT and Mixing but not VST. To see this, one shows that if $v$ is a word in the language of $\Omega$ and $x \in \Omega$
there exists $N$, depending on $v$ and $x$, such that for every $n \geq N$ there exists a word $a$ of length $n$ such that
$vax \in \Omega$. This implies SPT and Mixing. On the other hand, as $x$ varies, the minimum length $N$ is not bounded
and this implies that the system is not VST.

 \vskip .3cm

5. Let $\S = \{ a_i, b_i : i \in \Z \}$ consist of two bi-infinite sequences
in $(0,1)$, with $b_{i - 1} < a_i < b_i < a_{i+1}$ for all $i \in \Z$ with $a_i, b_i \to 0, 1$ as $i \to \mp\infty$. 
Let $J^+_i = [a_i,b_i], \ J^-_i = [b_i,a_{i+1}]$ for $i \in \Z$. with $J^{\pm, \circ}_i$ the corresponding open intervals.

Define: $f$ on $[0,1]$ so that
$$0 \mapsto 0, \ a_i \mapsto a_{i-1}, \ b_i \mapsto b_{i+1}, \ 1 \mapsto 1$$
and $f$ is linear on each $J^{\pm}_i$. The points $0$ and $1$ are fixed points with singleton pre-images.

Thus, $f$ is monotone increasing on each $J^+_i$, decreasing on each $J^-_i$ and the slope on each interval
has absolute value greater than one. Observe that $f$ is monotone on a closed interval $J$ iff $J^{\circ}$ is disjoint from
$\S$.

Clearly,
$$f(J^{-}_{i}) \ \supset \ J^+_i$$
and
$$f([a_{i-k},b_{i+k}]) \ \supset \ [a_{i-(k+1)},b_{i+(k+1)}] \ \supset \ [a_{i-k},b_{i+k}]$$
for all $i \in \Z$ and $k = 0,1,2, \dots$.

Thus, for each $i \in \Z$ the sequence of intervals  $\{ f^k(J^+_i) : k = 0, 1, \dots \}$ is increasing with union $(0,1)$.

\begin{lem}\label{AnimaExample} If $J,K$ are closed, nontrivial intervals
in $(0,1)$ then there exists $N \in \N$ so that $f^k(J) \supset K$ for
all $k \geq N$. \end{lem}

{\bfseries Proof:} If some iterate of $J$ contains some $b_i$ and $a_j$ then it contains either $J^+_i$ or $J^-_i$ and
the result then follows using the increasing sequence $\{ f^k(J^+_i)\}$. If it contains either two $b_i$'s or two $a_j$'s
then it contains some $b_i$ and some $a_j$. Thus, it suffices to show that some iterate of $J$ meets $\S$ in at least two points.

First we show that if $G$ is a nonempty open interval in $(0,1)$ then some iterate of $G$ meets $\S$. Assume not.
Then, inductively, each $f^k(G)$ is an  interval contained in  some $J^{\pm, \circ}_i$ and so it is an open interval
on which $f$ is monotone. As the slopes have absolute value greater than one, the lengths are increasing.  On the
other hand, for any $\epsilon > 0$ there are only finitely many $J^{\pm}_i$ with length greater than or equal to $\epsilon$.
So with $\epsilon$ equal to the length of $G$,  each $f^k(G)$ is contained in one of this finite collection. On a finite
collection of $J^{\pm}_i$'s the absolute value of the slopes is bounded above one.  Hence, the lengths must be increasing
at least at a geometric rate and this is impossible in $[0,1]$.

Thus, we obtain that some iterate of $G$ contains a nontrivial closed interval $[c,d]$ with either $c$ or $d$ in $\S$.
By the previous argument,  $f^{k}((c,d))$ meets $\S$ for some $k = 0, 1, \dots$. Let $k$ be the smallest such nonnegative
integer. Then on $[c,d]$ the map $f^k$ is injective. Hence, the point $e \in f^k((c,d)) \cap \S$ is not equal to
either $f^k(c)$ or $f^k(d)$ one of which also lies in $\S$.

Thus, as required we have shown that some iterate of a nontrivial interval meets $\S$ in at least two places.

It follows that $f$ is exact transitive and topologically mixing. Since, $\O^-(0) = \{ 0 \}$ and $\O^-(1) = \{ 1 \}$,
the system is not ST.

 \vskip .3cm

6. It is well-known that there are weak mixing homeomorphisms which are not mixing.  For shift examples see  \cite{ag}.

$\Box$

We are left with the following questions:
\begin{itemize}
\item Does Strongly Exact Transitive $\Rightarrow$ Strongly Product Transitive?

\item Does Strongly Product Transitive $\Rightarrow$ Mixing?

\item Does Exact Transitive $\Rightarrow$ Mixing?
\end{itemize}



Our example in 5. above yields  a
transitive map on $[0,1]$  such that
 the two endpoints have singleton backward orbits; whereas  backward
 orbits of all other points are dense. This is similar to an example given in \cite{b-m}.

 The same example when restricted to the open
interval $(0,1)$ can be used to
 provide a transitive map on ${\R}$ with all points having dense backward orbits.  An example of a
 transitive map on ${\R}$ such that $f^{-1}(0)= \{0\}$ and all other
 backward orbits are dense has been constructed in
 \cite{as}. It has been also proved that there can be at most one such point,
 whose backward orbit is not dense, for any transitive map on $\R$ in  \cite{avs}.

Towards the end, it would also be interesting to look into noninvertible, minimal subshifts.
Observe that given any minimal subshift $\Psi \subset \mathcal{A}^{\Z}$, where $\mathcal{A}$ is a
finite alphabet, we can define a  subshift $\Omega \subset \mathcal{A}^{\N}$, by
taking the projection $\pi: \Psi \to \Omega$ induced by the inclusion of $\N$ into $\Z$.
As a factor of a minimal system, $(\Omega, \sigma)$ is minimal. If this system is infinite,
then it will contain asymptotic pairs of points, see, e.g. \cite{aus} page 19.
%
%
Beginning with such a pair and shifting we can obtain
  $u, v \in \Omega$ such that $u_0 \neq v_0$,
 but $u_n = v_n$ for all $n \in \N$.
Thus, $\sigma$  is not invertible on $\Omega$.

We  refer the reader to \cite{k2} for more on noninvertible minimal maps.

\bibliographystyle{plain}

\end{document}